\documentclass[11pt,reqno]{amsart}

\usepackage{lineno}
\usepackage{amsmath,amsfonts,amsthm}
\usepackage{amssymb,latexsym}
\usepackage{cases}
\usepackage{booktabs}
\usepackage{longtable}

\usepackage{cleveref}

\usepackage[a4paper, total={16.5cm,25cm}]{geometry}
\renewcommand\appendix{\setcounter{secnumdepth}{-2}}

\usepackage{url,cancel}
\usepackage{float} 

\usepackage{caption,enumitem,bm,comment,xcolor}
\usepackage{mathrsfs,mathtools}
\setlist[enumerate]{leftmargin=*,label=\rm{(\arabic*)}}

\newtheorem{thm}{Theorem}[section]

\newtheorem{lem}[thm]{Lemma}
\newtheorem{prop}[thm]{Proposition}
\newtheorem{cor}[thm]{Corollary}
\newtheorem{re}[thm]{Remark}
\newtheorem{remarks}[thm]{Remarks}
\theoremstyle{remark}

\theoremstyle{definition}
\newtheorem*{re*}{Remark}
\newtheorem*{remarks*}{Remarks}

\newcommand{\ord}{\operatorname{ord}}

\newcommand{\leg}[2]{\left(\frac{#1}{#2}\right)}

\newcommand{\s}{\sigma}
\renewcommand{\a}{\alpha} 
\renewcommand{\b}{\beta} 
\newcommand{\g}{\gamma}
\newcommand{\e}{\varepsilon}
\renewcommand{\d}{\delta}
\renewcommand{\l}{\lambda}

\newcommand{\Om}{\Omega}
\newcommand{\De}{\Delta}

\newcommand{\Z}{{\mathbb Z}} 
\newcommand{\Q}{{\mathbb Q}} 
 
\newcommand{\F}{{\mathbb F}} 
\newcommand{\R}{{\mathbb R}}

\newcommand{\cG}{{\mathcal G}} 
 
\newcommand{\N}{\mathbb{N}}

\newcommand{\GL}{\operatorname{GL}}
\newcommand{\SO}{\operatorname{SO}}

\newcommand{\cond}{c}
\newcommand{\condd}{d}
\newcommand{\sm}{\setminus}

\renewcommand{\pmod}[1]{\ \left( \mathrm{mod} \, #1 \right)}

\newcommand{\Pmod}[1]{\ ( \mathrm{mod} \, #1 )}
\newcommand{\ol}[1]{\overline{#1}}

\newcommand{\sangle}[1]{\langle #1\rangle}
\newcommand{\pangle}[1]{\left\langle #1\right\rangle}

\newcommand{\Lr}[1]{\mathscr{L}^{[#1]}}

\usepackage[d]{esvect}

\makeatletter
\@addtoreset{equation}{section}
\makeatother
\numberwithin{equation}{section}
\numberwithin{table}{section}

\author{Kathrin Bringmann }
\address{University of Cologne, Department of Mathematics and Computer Science, Weyertal 86-90, 50931 Cologne, Germany}
\email{kbringma@math.uni-koeln.de}
\author{Zilong He }
\address{School of Computer Science and Technology, Dongguan University of Technology, Dongguan 523808, China}
\email{zilonghe@connect.hku.hk}
\author{Ben Kane}
\address{Department of Mathematics, University of Hong Kong, Pokfulam, Hong Kong}
\email{bkane@hku.hk}
\title[Class numbers of inhomogenous quadratic polynomials]{The mass of shifted lattices and class numbers of inhomogeneous quadratic polynomials}
\thanks{ }
\subjclass[2020]{11E08, 11E12, 11E41}
\date{\today}
\keywords{class numbers of shifted lattices, mass formula, quadratic polynomials}

\begin{document}
\begin{abstract}
	In this paper, we investigate class numbers of shifted quadratic lattices $L+\frac{\bm{u}}{\cond}$ with $\bm{u}\in L$ and odd conductor $\cond\in \N$. For a lattice $L$ whose genus only contains one class, we determine a lower bound for the number of classes in the genus of $L+\frac{\bm{u}}{\cond}$ depending on $\cond$. As a result, we obtain an explicit bound $\cond_0$ such that any such shifted lattice with one class in its genus must have conductor smaller than $\cond_0$, restricting the possible choices of such $L+\frac{\bm{u}}{\cond}$ to a finite set. 
\end{abstract}
\maketitle

\section{Introduction and statement of results}\label{sec:intro}

Positive-definite integral quadratic forms of rank $\ell$ over $\Z$ naturally split into equivalence classes under isometry (see Subsection \ref{sec:isometry}). For $\ell=2$, Lagrange \cite{Lagrange} showed that there exist only finitely many classes of binary quadratic forms of a given discriminant $D<0$. Gauss \cite[Section V, Articles 303--304]{Gauss} then defined the {\it class number} $h(D)$ as the number of classes with discriminant $D$ and for any $n\in\N$ asked for a classification of all $D<0$ with $h(D)=n$. He claimed that this set is finite because $h(D)\to\infty$ as $|D|\to\infty$ \cite[Section V, articles 303--304]{Gauss}. Gauss's claim was proven by Heilbronn \cite{Heilbronn}. Siegel \cite{siegel_anzahl_1936} then showed that $h(D)\gg|D|^{\frac12-\e}$ for any $\e>0$, but the lower bound was ineffective, so it did not resolve Gauss's class number problem for any given $n$. However, weaker effective bounds were later found and the case $n=1$ was independently solved by Stark \cite{Stark} (following ideas of Heegner) and Baker \cite{Baker}, with subsequent work culminating in the solution of Gauss's class number problem for $n\le100$ \cite{Baker2,Oesterle,Stark2,Watkins}.

For $\ell>2$, the discriminant is replaced by the determinant of the Gram matrix and one can similarly count the number of classes of a given determinant, which is again finite \cite[Chapter 9, Theorem 1.1]{Cassels}. The set of positive-definite quadratic forms with a given determinant further splits into genera, which consists of those forms that are $p$-adically isometric for every prime $p$. For a quadratic form $Q$, $h_Q$ counts the number of classes in the genus of $Q$ and one may investigate the analogous classification of genera with class number $h$. Watson \cite{WatsonSmallClassNumber} proved that $n=1$ can only occur for $\ell\le10$ and fully classified the case $n=1$ for $3\le\ell\le10$ in a series of papers\footnote{There were some omissions and a corrected list was compiled by Kirschmer and Lorch \cite{KL1}.} \cite{Watson1,Watson3,Watson2}. The case $n=2$ was solved by Kirschmer and Lorch \cite{KL}; Magnus \cite{Magnus} and Pfeuffer \cite{Pfeuffer} gave general lower bounds on the class number in terms of the rank $\ell$ and determinant $D$.

It is also natural to restrict $\bm x\in\Z^\ell$ appearing in \eqref{eqn:Qdef} to certain arithmetic progressions. For a quadratic form $Q$ and arithmetic progressions $x_j\equiv h_j\Pmod{N_j}$, we call the triple $(Q,\bm h,\bm N)$ a {\it quadratic form with congruence conditions}. There is an analogous theory of isometries for such quadratic forms with given congruence conditions. Hence we have a well-defined class number $h_{Q,\bm h,\bm N}$ for quadratic forms with congruence conditions $(Q,\bm h,\bm N)$ as well. For fixed $Q$, it is natural to investigate how the class number grows as a function of the {\it conductor}, which is the minimal $\cond\in\N$ such that $h_j\cond\equiv0\Pmod{N_j}$ for all $1\le j\le\ell$. The growth of these class numbers, if $\cond$ is an odd prime power with a certain technical condition depending on $Q$ and $\ell\ge3$, was studied by Sun \cite{sun_class_2016,sun_growth_2017}. Specifically, letting $d_Q$ denote the determinant of $Q$ (see \eqref{eqn:dLdef}), Sun assumed that $\gcd(d_Q,\cond)=1$. Denoting by $m_Q$ the Siegel mass of $Q$ (see \eqref{eqn:SiegelMassQ} and \eqref{eqn:SiegelMass}), we build on her work to obtain a lower bound for the class number if $\ell\ge3$ under a technical assumption.
\begin{thm}\label{thm:classbound}
	Suppose that $\ell\ge3$ and $(Q,\bm h,\bm N)$ is a positive-definite diagonal quadratic form with congruence conditions having odd conductor $\cond$ such that $\gcd(d_Q,Q(\bm{h}),\cond)=1$. Then, for each $\d>0$, there exists an explicit constant $C_\d>0$ (see \eqref{eqn:Cdeltadef}) such that
	\[
		h_{Q,\bm h,\bm{N}} \ge m_Q C_\d  \cond^{\ell-1-\d}.
	\]
\end{thm}

We next investigate the class number one problem for genera of quadratic forms with congruence conditions. To be more precise, one can use \Cref{thm:classbound} to restrict to those $Q$, $\bm h$, and $\bm N$ for which $h_{Q,\bm h,\bm N}=1$; such choices of $Q$, $\bm h$, and $\bm N$ have class number one. By \cite[Corollary 4.4]{chan_representations_2013}, $h_{Q,\bm h,\bm N}\ge h_Q$, so we may restrict ourselves to those class number one quadratic forms $Q$ classified by Watson.

\begin{cor}\label{cor:ClassNumberOne}
	Let $\bm h,\bm N\in\N^\ell$, and $Q$ is a positive-definite integral quadratic form of rank $\ell\ge3$ for which $(Q,\bm h,\bm N)$ has class number one\footnote{For each $3\le\ell\le10$, explicit bounds $\cond_\ell$ such that $\cond\le\cond_\ell$ are given in Table \ref{tab:ClassNumberOne}.}. If the conductor $\cond$ of $(Q,\bm h,\bm N)$ is odd and $\gcd(d_Q,Q(\bm{h}),\cond)=1$, then $\cond\le 81$.
\end{cor}

\begin{remarks} \label{re:ClassNumberOne}
\noindent

\noindent
\begin{enumerate}
\item 	One motivation for classifying forms with (genus) class number one is a certain weighted average of Siegel \cite{siegel_uberdie_1935,SiegelIndefinite1,SiegelIndefinite2} and Weil \cite{Weil}. For $m\in\N$, they showed an explicit formula for the (weighted) average number of solutions to $Q(\bm x)=m$ with $Q$ running through representatives for the isometry classes in a given genus of integral quadratic forms. If the genus of $Q$ contains only one class, then this yields an identity for the number of solutions to $Q(\bm x)=m$. In this framework, a famous identity of Gauss (see e.g. \cite[Theorem 8.5]{OnoBook}) relating the number of solutions to $\sum_{j=1}^3x_j^2=m$ with class numbers of binary quadratic forms can be explained by the fact that the quadratic form $Q(\bm x)=\sum_{j=1}^3x_j^2$ has class number one.
	A number of similar formulas were obtained by two of the authors \cite{bringmann_class_2020} for counting solutions to $Q(\bm x)=m$ with congruence conditions $\bm x\equiv\bm h\Pmod{N}$ (for a fixed $N\in\N$). Since one would not expect such formulas to hold very often if the class number is not equal to one, Corollary \ref{cor:ClassNumberOne} yields a natural restricted search space for finding further identities.
\item
The class number of $(Q,\bm{h},\bm{N})$ is at least as large as the class number of $Q$, so $Q$ has class number one if $(Q,\bm{h},\bm{N})$ has class number one. By the classification of Kirschmer and Lorch \cite{KL1}, there are only finitely many choices of $Q$ and these may be explicitly listed. Moreover, for each $Q$, there are only finitely many choices of $\bm{h}$ and $\bm{N}$ with $\cond\leq 81$, so \Cref{cor:ClassNumberOne} implies that there are only finitely many triples $(Q,\bm{h},\bm{N})$ with class number one and $\gcd(2d_Q,2Q(\bm{h}),\cond)=1$. Since the bound $\cond\leq 81$ on the conductor is relatively small, one can easily use a computer to check all possible triples and give a full list of such class number one quadratic forms with congruence conditions for which $\gcd(2d_Q,2Q(\bm{h}),\cond)=1$. Indeed, if $\ell=3$ an algorithm to compute the genus of a given quadratic form with congruence conditions was given by Singhal \cite{Singhal} and hence for $\ell=3$ one only needs to run the provided code for each possible cases to obtain a full list for the ternary case. A generalization of \cite[Proposition 4.3]{Singhal} to higher rank should lead to an algorithm that can be used for $4\leq \ell\leq 10$. A more naive and slower algorithm involves simply listing all possible congruence conditions for each $Q$ and grouping them into genera by checking local conditions at every prime; although slower and less practical, this certainly work for all ranks. The computer computations necessary to list all such class number one forms appears to be a moderately-long, but feasible, calculation, due to the large number of class number one quadratic forms, so we have not worked out the details. However, for a given quadratic form, the calculation should be fairly straightforward, and it may be interesting to give a complete list.
\item We note that preliminary calculations indicate that the bound on $\cond$ is significantly worse if the condition $\gcd(2d_Q,2Q(\bm{h}),\cond)=1$ is dropped, at least with the techniques used in this paper. While a technical modification of the proof may lead to some finite bound, the primes dividing $\gcd(2d_Q,2Q(\bm{h}),\cond)$ would likely appear to some rather high power in the bound on $\cond$, so the resulting finite set of possible class number one quadratic forms with congruence conditions would be infeasible to check for a full enumeration of all class number one quadratic forms with congruence conditions.
\end{enumerate}
\end{remarks}

In order to study the number of classes in \Cref{thm:classbound}, we consider the Siegel mass. Following Minkowski's study on the mass for lattices, Siegel \cite{siegel_uberdie_1935,siegel_uberdie_1936,siegel_uberdie_1937} proved the mass formula, which expresses the mass as an infinite product of local factors over rational primes. Since the mass is a weighted sum over genus representatives, one obtains important information about the genus by evaluating and bounding these local factors, which are easier to compute because they only depend on $p$-adic properties of the lattice. Based on the notion of maximal lattices introduced by Eichler (see \cite[p.\hskip 0.1cm 1]{shimura_anexact_1999}), Shimura \cite{shimura_anexact_1999} pointed out that the mass formula for arbitrary lattices can be obtained via a relation (see \cite[Lemma 8.7 and the discussion at the end of Section 8]{shimura_anexact_1999}) with the mass formula for maximal lattices. Shimura then gave an exact mass formula for maximal lattices over a totally real number field \cite[Theorem 5.8]{shimura_anexact_1999}. Hanke \cite{hanke_anexact_2005} later extended Shimura's formula to arbitrary number fields. Recently, Sun formulated the mass formula for shifted lattices, and then used it to compute the class number of the weighted sums of ternary polygonal numbers \cite{sun_class_2016} and to study the growth of the class numbers of inhomogeneous quadratic polynomials over a totally real number field \cite{sun_growth_2017} if the conductor is odd. As a primary step towards \Cref{thm:classbound}, we extend Sun's work and also obtain explicit constants.

\begin{thm}\label{thm:diagonalclassformulaintro}
	Let $V$ be an $\ell$-dimensional $\Q$-vector space equipped with a positive-definite quadratic form, $L\subset V$ be a $\Z$-lattice corresponding to a positive-definite diagonal quadratic form $Q$, $\bm u\ne\bm0$ a primitive vector in $L$, and $\cond\in\N$ odd with $\gcd(d_L,Q(\bm{u}),\cond)=1$. Then there exists $C_\d$ (see \eqref{eqn:Cdeltadef}) such that
	\begin{equation*}\label{eqn:massratiobound}
		\frac{m\left(L+\frac{\bm u}{\cond}\right)}{m(L)} \ge C_\d \cond^{\ell-1-\d}.
	\end{equation*}
\end{thm}

\begin{re}
	For each $p\mid c$ (see \eqref{eqn:massrelationZilong}) the left-hand side of \Cref{thm:diagonalclassformulaintro} only depends on $L_p:=L\otimes\Z_p$. The assumption that $Q$ is diagonal is used in the proof of \Cref{thm:diagonalclassformulaintro} to bound the factors appearing for each prime $p\mid\cond$ on the right-hand side of \eqref{eqn:massrelationZilong}. The proof only requires the fact that $L_p$ is diagonalizable for every prime $p\mid\cond$. This is always the case for $p>2$ (see \cite[Subsection 4.3]{WatsonIntegral}). Thus the proof only uses the assumption that $L_2$ is diagonalizable. If $L$ is maximal, then, by the exact formula of $m(L)$ determined by Shimura \cite[Theorem 5.8]{shimura_anexact_1999}, one can bound the mass $m(L+\frac{\bm u}{\cond})$ from below by Theorem \ref{thm:diagonalclassformulaintro}.
\end{re}

The paper is organized as follows. In Section \ref{sec:notation}, we list notation used throughout the paper. In Section \ref{sec:prelim}, we recall some facts about shifted lattices and quadratic forms with congruence conditions. In Section \ref{sec:associates}, we give conditions for vectors to be associates over $\Z_p$. Lower bounds for the ratios of masses between lattices and shifted lattices are obtained in Section \ref{sec:mass}. Finally, the proofs of Theorem \ref{thm:classbound}, Corollary \ref{cor:ClassNumberOne}, and Theorem \ref{thm:diagonalclassformulaintro} are given in Section \ref{sec:proofmainresult}.

\section*{Acknowledgements}

The first author is supported by the Deutsche Forschungsgemeinschaft (DFG) Grant No. BR 4082/5-1. Part of the research was conducted while the second author was a Ph.D. student at the University of Hong Kong. The research of the third author was supported by grants from the Research Grants Council of the Hong Kong SAR, China (project numbers HKU 17303618 and 17314122). The authors thank Billy Chan and Rainer Schulze-Pillot for helpful comments.

\section{Notation}\label{sec:notation}

In this section we list commonly used notation for the convenience of the reader. 
\bgroup
\def\arraystretch{1.7}
\begin{longtable}{|c|c|}
\hline
notation&meaning/definition\\
\hline
$V$&quadratic space over a field $\F$ ($\F=\Q$, $\F=\Q_p$, or $\F=\R$),\\
& \hspace{-2.55cm}$\begin{array}{l}{}\\\vspace{-2cm}\\\text{endowed with a quadratic form }Q:V\to\F\end{array}$\\
\hline
$B(\bm{u},\bm{v})$& symmetric bilinear form $Q(\bm{u}+\bm{v})-Q(\bm{u})-Q(\bm{v})$\\
\hline
$L$ &$R$-lattice in $R=\Z$ or $R=\Z_p$\\
\hline
$d_L$&determinant of $L$ (see \eqref{eqn:dLdef})\\
\hline
$L+\frac{\bm{u}}{\cond}$ &shifted lattice ($\bm{u}\in L$, $\cond\in\N$)\\
\hline
$s(Q)$, $s(L)$&scale of $Q$, $L$ (see \eqref{eqn:scaledef} and \eqref{eqn:scaleLdef})\\
\hline
$n(Q)$, $n(L)$&norm of $Q$, $L$  (see \eqref{eqn:normdef} and \eqref{eqn:scaleLdef})\\
\hline
$O(V)$&isometries of $V$ (see \eqref{eqn:isometry})\\
\hline
$O\left(L+\frac{\bm{u}}{\cond}\right)$&isometries sending $L+\frac{\bm{u}}{\cond}$ to itself\\
\hline
$\SO(V)$, $\SO\left(L+\frac{\bm{u}}{\cond}\right)$&special orthogonal groups (see \eqref{eqn:SOdef})\\
\hline
$Q_L$& quadratic form associated to $L$ and a choice of generators (see \eqref{eqn:QLdef})\\
\hline
$\Q_p$&$p$-adic rationals\\
\hline
$\Z_p$&$p$-adic integers\\
\hline
$V_p$ &localization of $V$ from $\Q$ to $\Q_p$\\
\hline
$h_{L+\frac{\bm{u}}{\cond}}$&class number of $L+\frac{\bm{u}}{\cond}$ (see \eqref{eqn:classnumdef})\\
\hline
$h_{L+\frac{\bm{u}}{\cond}}^+$&proper class number of $L+\frac{\bm{u}}{\cond}$ (see \eqref{eqn:classnumdef})\\
\hline
$m\left(L+\frac{\bm{u}}{\cond}\right)$&Siegel mass of $L+\frac{\bm{u}}{\cond}$ (see \eqref{eqn:SiegelMass})\\
\hline 
$m^+\left(L+\frac{\bm{u}}{\cond}\right)$&proper Siegel mass of $L+\frac{\bm{u}}{\cond}$ (see \eqref{eqn:ProperSiegelMass})\\
\hline
$v(x)=v_p(x)$& $p$-adic order of $x\in\Q_p$\\
\hline
$\ord_p(n)$ &$p$-adic order of $n\in\Z$\\
\hline
$v(L;\bm{x})=v_p(L;\bm{x})$&$v(L;\bm{x}):=\min_{\bm{y}\in L} v\left(\frac{1}{2}B\left(\bm{x},\bm{y}\right)\right)$\\
\hline
$\Lr{r}$&$\Lr{r}:=\{\bm x\in L:v(L;\bm x)\ge r\}$\\
\hline
$\bm{u_p}$&$\bm{u_p}:=d^{-1}\bm{u}\in L_p$ for $\cond=p^kd$ with $p\nmid d$, $\Z$-lattice $L$, and $\bm{u}\in L$\\
\hline
$\Omega_{\cond}$ & $\Om_\cond:=\{p\text{ prime}:p\mid \cond\}$\\
\hline
$\Omega_{1,\cond}(L,\bm{u})$ & $\Om_{1,\cond}(L,\bm u):=\{p\in \Omega_{\cond}:Q(\bm{u})\in \Z_p^{\times}\text{ or }(L_p\text{ is unimodular and }Q(\bm{u})\in p\Z_p)\}$\\
\hline
$\Omega_{2,\cond}(L,\bm{u})$ &$p\in \Omega_{\cond}$ with $L_p$ not unimodular and $Q(\bm{u})\in p\Z_p$\\
\hline
$R\left(Q,\bm{u},p^k\right)$&$R\left(Q;\bm u,p^k\right):=\left\{\bm x\in L_p/p^kL_p:\bm x\notin pL_p,Q(\bm x)\equiv Q(\bm u)\pmod{p^k}\right\}$\\
\hline
$r\left(Q,\bm{u},p^k\right)$&$r\left(Q,\bm{u},p^k\right):=\left|R\left(Q,\bm{u},p^k\right)\right|$\\
\hline
$R\left(Q,\bm{u},r,p^k\right)$&$R\left(Q;\bm u,r,p^k\right):=\left\{\bm x\in R\left(Q;\bm u,p^k\right):\bm x\equiv\bm u\pmod{p^r}\right\}$\\
\hline
$r\left(Q,\bm{u},r,p^k\right)$&$r\left(Q,\bm{u},r,p^k\right):=\left|R\left(Q,\bm{u},r,p^k\right)\right|$\\
\hline
$d_{\bm{u},p}=d_{\bm{u},p}(Q)$&depth $d_{\bm{u},p}(Q):=\min_{j\in I_{\bm{u},p}}\left(\ord_{p}(a_j)+1\right)$\\
\hline
$R_S\left(Q;\bm u,p^k\right)$&$R_S\left(Q;\bm u,p^k\right):= \{\bm x\in\Z^\ell/p^k\Z^\ell : Q(\bm x)\equiv Q(\bm u)\pmod{p^k}$\\
&$\hspace{2.1in}\begin{array}{l}{}\\\vspace{-1.9cm}\\\text{ and }p\nmid x_j\text{ for all }j\in S\}\end{array}$\\
\hline
$r_{S}\left(Q;\bm{u},p^{k}\right)$ &$r_{S}\left(Q;\bm{u},p^{k}\right):=\# R_{S}\left(Q;\bm{u},p^{k}\right)$\\
\hline
$I_{p}=I_{\bm{a},p}$& $I_{\bm{a},p}:=\{1\leq j\leq \ell :p\nmid a_j\}$\\
\hline
$t_p$&$t_p:=\#I_p$\\
\hline
$\De_p$&$\De_p:=\leg{-1}{p}p$, where $\leg{-1}{p}$ is the Legendre symbol\\
\hline
$s_{J,p}=s_{\bm{a},J,p}$&$s_{\bm{a},J,p}:=\#\{j\in J: p\nmid a_j\}$\\
\hline
$r_{J,p}$&$r_{J,p}:=\# \left\{j\in J : \left(\tfrac{a_j}{p}\right)=-1\right\}$\\
\hline
$C_p(m,t_p,r_p)$&see \eqref{eqn:Cpmtprp}\\
\hline
$\cond_p$& $\cond_p:=\ord_p(\cond)$\\
\hline
$f_p=f_p(\cond_p)$&see \eqref{eqn:fpdef}\\
\hline
\end{longtable}
\egroup

\section{Preliminaries}\label{sec:prelim}

\subsection{Quadratic forms}\label{sec:quadraticforms}

Let $\F$ be a local or global field with characteristic different from $2$ and ring of integers $R$. For $\ell\in\N$, an {\it integral\hspace{.05cm}\footnote{One calls these quadratic forms ``integral'' because the corresponding Gram matrix has entries in $R$. If one relaxes the condition on $a_{jk}$ to $a_{jk}\in R$, then one gets $R$-valued quadratic forms, i.e., $Q(\bm x)\in R$ for $\bm x\in R$.} quadratic form of rank $\ell$ over $R$} is a homogeneous polynomial with $a_{jk}\in R$ and $a_{jk}\in 2R$ if $j\ne k$
\begin{equation}\label{eqn:Qdef}
	Q(\bm{x}):=\sum_{1\leq j\leq k\leq \ell} a_{jk} x_j x_k.
\end{equation}
A quadratic form $Q_1$ is {\it equivalent} to $Q$ over $R$ if there exists $M\in\GL_\ell(R)$ such that $Q_1(\bm{x})=Q(M\bm{x})$.
In particular for $m\in R$ the equation $Q(\bm x)=m$ with $\bm x\in R$ has the same number of solutions as $Q_1(\bm x)=m$ with $\bm x\in R$ because $M$ is invertible over $R$. If $Q(\bm x)>0$ for all $\bm x\in R^\ell\sm\{\bm0\}$, then $Q$ is {\it positive-definite}. The {\it symmetric bilinear form} on $R^\ell$ associated to $Q$ is defined by
\[
	B(\bm{x},\bm{y}):=Q(\bm{x}+\bm{y})-Q(\bm{x})-Q(\bm{y}).
\]
Taking $\bm x=\bm y$ gives the relation $Q(\bm x)=\frac12B(\bm x,\bm x)$. The assumption that $Q$ is integral is equivalent to assuming that the {\it scale of $Q$}, defined by (see \cite[82.8 (1)]{omeara_introduction_2000})
\begin{equation}\label{eqn:scaledef}
	s(Q) := \pangle{\tfrac12B(\bm x,\bm y) : \bm x,\bm y\in R^\ell} \subseteq R.
\end{equation}
Here $\sangle S$ means the ideal generated by $S$ over the ring $R$. Similarly, the {\it norm of $Q$} is the ideal
\begin{equation}\label{eqn:normdef}
	n(Q) := \pangle{Q\left(\bm x\right) : \bm x\in \Z^{\ell}}\subseteq s(Q).
\end{equation}
We assume throughout the paper that $s(Q)=R$ and are mainly interested in the case $R=\Z$.

\subsection{Quadratic lattices and isometry classes}\label{sec:isometry}

Fix an $\ell$-dimensional vector space $V$ over $\F$, called {\it quadratic space}, with associated symmetric bilinear form $B(\bm u,\bm v)$ from $V\times V$ to $\F$. Then $Q(\bm u):=\frac12B(\bm u,\bm u)$ is a quadratic form. A vector $\bm u\in V\sm\{0\}$ is {\it isotropic} if $Q(\bm u)=0$ and {\it anisotropic} otherwise. An {\it isometry of $V$} is a linear isomorphism $\s:V\to V$ satisfying, for every $\bm u,\bm v\in V$,
\begin{equation}\label{eqn:isometry}
	B\left(\sigma(\bm{u}),\sigma(\bm{v})\right)=B(\bm{u},\bm{v}),
\end{equation}
we denote the group of isometries of $V$ by $O(V)$. We use $O(V)$ to form equivalence classes for $R$-lattices $L\subseteq V$ and the {\it shifted lattices} $L+\frac{\bm u}{\cond}$, where $\bm u\in L$ and $\cond\in\N$. Note that if $\bm u\in \cond L$, then $L+\frac{\bm u}{\cond}=L$, so we may restrict to shifted lattices. Shifted lattices $L+\frac{\bm u}{\cond}$ and $K+\frac{\bm v}{\condd}$ are {\it isometric} (written $L+\frac{\bm u}{\cond}\cong K+\frac{\bm v}{\condd}$) if there exists $\s\in O(V)$ such that $K+\tfrac{\bm{v}}{\condd} =\sigma\left(L+\tfrac{\bm{u}}{\cond}\right)$.
If $\s$ is a bijection from $L+\frac{\bm u}{\cond}$ to itself, then we call $\s$ an {\it isometry of $L+\frac{\bm u}{\cond}$} and the isometry group of $L+\frac{\bm u}{\cond}$ is denoted by $O(L+\frac{\bm u}{\cond})$. We moreover define the {\it special orthogonal groups}
\begin{equation}\label{eqn:SOdef}
	\SO(V):=\{\sigma\in O(V):\det(\sigma)=1\}\quad\text{ and }\quad \SO\left(L+\tfrac{\bm{u}}{\cond}\right):=\left\{\sigma\in O\left(L+\tfrac{\bm{u}}{\cond}\right): \det(\sigma)=1\right\}.
\end{equation}
Note that 
\begin{equation}\label{eqn:OSOindex}
	\left[O\left(L+\tfrac{\bm u}{\cond}\right) : \SO\left(L+\tfrac{\bm u}{\cond}\right)\right] =
	\begin{cases}
		2 & \text{if there exists }\s\in O\left(L+\frac{\bm u}{\cond}\right),\ \det(\s)=-1,\\
		1 & \text{otherwise}.
	\end{cases}
\end{equation}
In particular,
\begin{equation}\label{eqn:OSOinequality}
	\left[O(L) : O\left(L+\tfrac{\bm u}{\cond}\right)\right] \le 2\left[\SO(L) : \SO\left(L+\tfrac{\bm u}{\cond}\right)\right].
\end{equation}
Elements $\bm u,\bm v\in L$ are {\it associated} if there exists $\s\in O(L)$ for which $\s(\bm u)=\bm v$. In this case, we write $\bm u\sim_L\bm v$. We omit the dependence on $L$ if it is clear from the context.

\subsection{Quadratic lattices and quadratic forms}\label{sec:LatticeQuadForms}

For a lattice $L\subset V$, we choose generators $\bm{u_1},\dots,\bm{u_\ell}\in L$ and associate a quadratic form $Q_L$ to $L$ by setting, for $\bm x\in R^\ell$, 
\begin{equation}\label{eqn:QLdef}
	Q_L(\bm{x})=Q_{L,\bm{u_1},\dots,\bm{u_{\ell}}}(\bm{x}):=Q\left(\sum_{j=1}^\ell x_j \bm{u_j}\right)=\tfrac12\sum_{j=1}^{\ell}\sum_{k=1}^{\ell} B\left(\bm{u_j},\bm{u_k}\right)x_jx_k. 
\end{equation}
Although $Q_L$ depends on the choice of generators, a different choice yields an equivalent quadratic form over $R$.
We therefore assume that the set of generators is fixed. One also sees directly that
\[
	\#\left\{\bm{x}\in R^{\ell}: Q_{L}(\bm{x})=m\right\} = \#\left\{\bm{u}\in L: Q(\bm{u})=m\right\}
\]
is independent of the choice of generators. Letting $G_L$ be the {\it Gram matrix} of $L$ with respect to the generators $\bm{u_1},\dots,\bm{u_{\ell}}$, whose $(j,k)$-th component is $\frac12B(\bm{u_j},\bm{u_k})$, we define the {\it determinant} of $L$ (and $Q_L$) by 
\begin{equation}\label{eqn:dLdef}
	d_{Q_L}:=d_L:=\det\left(G_L\right).
\end{equation}
The determinant is independent of the choice of generators, up to the square of a unit of $R$.
For $R=\Z$ the determinant is unique and for $R=\Z_p$ the quantities $\ord_p(d_L)$ and $(\frac{p^{-\ord_p(d_L)}d_L}{p})$ are well-defined. A lattice $L$ is called {\it unimodular (with respect to $R$)} if $d_L$ is a unit in $R$, and {\it integral} if $Q_L$ is integral. For an integral lattice, we define the {\it scale of $L$} and the {\it norm of $L$} by
\begin{equation}\label{eqn:scaleLdef}
	s(L) := s(Q_L),\qquad n(L) := n(Q_L).
\end{equation}

For quadratic spaces $V_j$ of dimensions $\ell_j$ $(1\leq j\leq k)$ with associated bilinear forms $B_j$, we define a quadratic space 
\[
	V=V_1\perp \cdots \perp V_k:=\left\{ \left(\bm{x}_1,\dots,\bm{x_k}\right): \bm{x_j}\in V_j\right\}
\]
of dimension $\sum_{j=1}^k \ell_j$, with component-wise addition and the associated bilinear form given by
\begin{equation}\label{eqn:orthogonalB}
	B\left(\left(\bm{x}_1,\dots,\bm{x_k}\right),\left(\bm{y}_1,\dots,\bm{y_k}\right)\right):=\sum_{j=1}^{k} B_j\left(\bm{x_j},\bm{y_j}\right).
\end{equation}
One naturally embeds $V_j$ into $V$ by setting $\bm{x_d}=\bm{0}$ for $d\neq j$; we abuse notation by writing $V_j$ for both the original quadratic space and its embedding. Note that $\bm{x_j}\in V_j$ and $\bm{x_d}\in V_d$ are orthogonal in $V$ if $j\neq d$, explaining the notation. Similarly, for lattices $L_j\subset V_j$ ($1\leq j\leq k$), write 
\begin{equation}\label{eqn:perpdef1}
	L_1\perp\dots\perp L_k := \{(\bm{x_1},\dots,\bm{x_k}) : \bm{x_j}\in L_j\} \subset V_1\perp\dots\perp V_k.
\end{equation}

From another perspective, suppose that $L_1,L_2,\dots,L_k$ are lattices contained in a quadratic space $V$ and $L_j$ and $L_r$ are orthogonal if $j\ne r$, i.e., $\bm{x_j}\in L_j$ and $\bm{x_r}\in L_r$ we have $B(\bm{x_j},\bm{x_r})=0$. We then abuse notation and also write
\begin{equation}\label{eqn:perpdef2}
	L_1\perp\cdots\perp L_k := \left\{\sum_{j=1}^k \bm{x_j} : \bm{x_j}\in L_j\right\}
\end{equation}
for the lattice spanned by $\{L_1,\dots,L_k\}$. The two definitions \eqref{eqn:perpdef1} and \eqref{eqn:perpdef2} are consistent in the sense that the lattice spanned by $\{L_1,\dots,L_k\}$ is a lattice in the quadratic space $\Q L_1\perp\cdots\perp\Q L_k$ and $\Q L_1\perp\cdots\perp\Q L_k$ naturally embeds into $V$ by the map $(a_1\bm{x_1},\dots,a_k\bm{x_k})\mapsto\sum_{j=1}^ka_j\bm{x_j}$. For $L_1\subseteq L$, we furthermore define the {\it orthogonal complement of $L_1$ in $L$} as
\[
	L_1^{\perp}:=\left\{\bm{u}\in L: B\left(\bm{u},\bm{u_1}\right)=0 \text{ for all } \bm{u_1}\in L_1\right\}.
\]
	
\subsection{Shifted lattices and quadratic forms}\label{sec:shifted}

Note that $L+\frac{\bm u}{\cond}$ only depends on $\bm u\in L/\cond L$.
For a shifted lattice $L+\frac{\bm u}{\cond}$, there similarly exists a corresponding quadratic form with congruence conditions (unique up to equivalence) which counts $\{ \bm{v}\in L+\tfrac{\bm{u}}{\cond}: Q(\bm{v})=m\}$. Choosing generators $\bm{u_1},\dots,\bm{u_\ell}$ of $L$, there is a one-to-one correspondence between $\bm v\in L+\frac{\bm u}{\cond}$ and $\bm x\in\Z^\ell$ given by $\bm v=\sum_{j=1}^\ell x_j\bm{u_j}+\frac{\bm u}{\cond}$. Thus
\begin{equation}\label{eqn:quadpoly}
	Q\left(\bm{v}\right)= \tfrac12\sum_{1\leq j,k\leq \ell} B\left(\bm{u_j},\bm{u_k}\right)x_jx_k + \sum_{1\leq j\leq \ell} \tfrac{x_j}{\cond} B\left(\bm{u_j},\bm{u}\right)+ \tfrac{1}{2\cond^2} B\left(\bm{u},\bm{u}\right).
\end{equation}
As noted in \cite[Section 4]{ChanRicci}, by completing the square, \eqref{eqn:quadpoly} becomes a quadratic form in a variable $\bm y$ with congruence conditions on $y_j$.

\subsection{Localization}

In this subsection, we restrict to $\F=\Q$. We associate with the {\it infinite prime} $\infty$ the usual absolute value $|x|_\infty:=|x|$. For a prime $p\in\N$ and $x\in\Q\sm\{0\}$, by unique prime factorization we have $x=p^r\frac ac$ for some $r\in\Z$ and $a,c\in\Z$ with $p\nmid ac$. We call $v_p(x):=r$ the {\it$p$-adic order} of $x$. For convenience, we set $v_p(0):=\infty$. The {\it$p$-adic absolute value} is $|x|_p:=p^{-v_p(x)}$. We let $\Q_p$ denote the $p$-adic completion of $\Q$. The {\it ring of $p$-adic integers} is defined as
\[
	\Z_{p}:=\left\{x\in \Q_{p}: v_p(x)\geq 0\right\}.
\]

With $\bm{u_1},\dots,\bm{u_n}$ a basis for a quadratic space $V$, for each finite prime $p$ we define the {\it localization of $V$ at the prime $p$} as
\[
	V_{p}:=V\otimes_{\Q} \Q_{p} = \bigoplus_{j=1}^\ell \Q_p \bm{u_j}.
\]
If $L\subset V$ is the lattice generated by $\bm{u_1},\dots,\bm{u_\ell}$ and $p$ is a finite prime, then the {\it localization of $L$ at $p$} is the $\Z_{p}$-lattice defined by 
\[
	L_{p}:=L\otimes_{\Z} \Z_{p}= \bigoplus_{j=1}^\ell \Z_{p} \bm{u_j}.
\]
The localization induces a (symmetric) bilinear form $B_p$ on $V_p$ by taking (for $\bm{x},\bm{y}\in \Q_p^\ell$)
\[
	B_p\left(\sum_{j=1}^\ell x_j\bm{u_j},\sum_{k=1}^\ell y_k\bm{u_k}\right) := \sum_{1\le j,k\le\ell} B\left(\bm{u_j},\bm{u_k}\right)x_jy_k.
\]
Note that if $\cond$ is invertible in $\Z_p$, then $L_p=\cond L_p$, and thus for any $\bm u\in L_p$ we have $\bm u\in\cond L_p$ and hence $L_p+\frac{\bm u}{\cond}=L_p$. At the infinite prime $\infty$, we analogously define 
\[
	V_{\infty}:=V\otimes_{\Q} \R = \bigoplus_{j=1}^{\ell} \R\bm{u_j}.
\]
The isometry class of a quadratic form over $\R$ is uniquely determined by the {\it signature} of $Q_L$, which is the pair $(s,t)$ such that $s$ (resp. $t$) is the largest dimension of the subspace of $V_\infty$ on which $Q_L$ is positive definite (resp. negative definite). We assume throughout that $Q_L$ is positive-definite.

\subsection{Classes and genera}\label{sec:PrelimClassesGenera}

Two shifted lattices $L+\frac{\bm{u}}{\cond}$ and $K+\frac{\bm{v}}{\condd}$ are in the same {\it class} (resp. in the same {\it proper class}) if there exists $\sigma\in O(V)$ (resp. $\sigma\in \SO(V)$) such that $\sigma\left(L+\tfrac{\bm{u}}{\cond}\right) = K+\tfrac{\bm{v}}{\condd}$. Moreover $L+\frac{\bm u}{\cond}$ and $K+\frac{\bm v}{\condd}$ are in the same {\it genus} (resp. {\it proper genus}) if $Q_L$ and $Q_K$ have the same signature and for every prime $p$ there exists $\s_p\in O(V_p)$ (resp. $\s_p\in\SO(V_p)$) such that $\sigma_p\left(L_p+\tfrac{\bm{u}}{\cond}\right) = K_p+\tfrac{\bm{v}}{\condd}$. Since $O(V)$ (resp. $\SO(V)$) embeds into $O(V_p)$ (resp. $\SO(V_p)$) for every prime $p$, the class (resp. proper class) is a subset of the genus (resp. proper genus). Hence it is natural to ask about the number of distinct classes (resp. proper classes) contained in the genus (resp. proper genus) of a given shifted lattice $L+\frac{\bm u}{\cond}$. We choose a set $\cG(L+\frac{\bm u}{\cond})$ (resp. $\cG^+(L+\frac{\bm u}{\cond})$) of representatives for the classes (resp. proper classes) in the genus (resp. proper genus). The {\it class number} $h_{L+\frac{\bm u}{\cond}}$ and {\it proper class number} $h_{L+\frac{\bm u}{\cond}}^+$ of $L+\frac{\bm u}{\cond}$ are then respectively defined by
\begin{equation*}\label{eqn:classnumdef}
	h_{L+\frac{\bm{u}}{\cond}}:=\#\cG\left(L+\tfrac{\bm{u}}{\cond}\right)\qquad\text{ and }\qquad h_{L+\frac{\bm{u}}{\cond}}^+:=\#\cG^+\left(L+\tfrac{\bm{u}}{\cond}\right).
\end{equation*}

\subsection{Siegel masses and the Siegel--Weil formula}

Since we assume throughout that $L$ is positive-definite, for a given shifted lattice $L+\frac{\bm{u}}{\cond}$ (with $\bm{u}\in L$ and $\cond\in\N$) the sets $O(L+\frac{\bm{u}}{\cond})$ and $\SO(L+\frac{\bm{u}}{\cond})$ are finite. In this case, the {\it Siegel mass} $m(L+\frac{\bm{u}}{\cond})$ and {\it proper Siegel mass} $m^+(L+\frac{\bm{u}}{\cond})$ of a shifted lattice $L+\frac{\bm{u}}{\cond}$ are defined by
\begin{align}\label{eqn:SiegelMass}
	m\left(L+\tfrac{\bm{u}}{\cond}\right)&:=\sum_{K+\frac{\bm{v}}{\condd}\in \cG\left(L+\frac{\bm{u}}{\cond}\right)} \frac{1}{\# O\left(K+\frac{\bm{v}}{\condd}\right)},\\
	\label{eqn:ProperSiegelMass}
	m^+\left(L+\tfrac{\bm{u}}{\cond}\right)&:=\sum_{K+\frac{\bm{v}}{\condd}\in \cG^+\left(L+\frac{\bm{u}}{\cond}\right)} \frac{1}{\#\SO\left(K+\frac{\bm{v}}{\condd}\right)}.
\end{align}
For the quadratic form with congruence conditions $(Q,\bm{h},\bm{N})$ related to $L+\frac{\bm{u}}{\cond}$ under the correspondence given in Subsection \ref{sec:shifted}, we also define 
\begin{equation}\label{eqn:SiegelMassQ}
	m_{Q,\bm h,N} := m\left(L+\tfrac{\bm u}{\cond}\right).
\end{equation}
We omit $\bm h$ and $N$ if $N=1$. Following \cite[proof of Theorem 3.6]{sun_growth_2017}), we have
\begin{equation}\label{eqn:massrelationZilong}
	m^+\left(L+\tfrac{\bm{u}}{\cond}\right)=m^+(L)\prod_{p}\left[\SO(L_p):\SO\left(L_p+\tfrac{\bm{u}}{\cond}\right)\right] = m^+(L)\prod_{p\mid \cond}\left[\SO(L_p):\SO\left(L_p+\tfrac{\bm{u}}{\cond}\right)\right].
\end{equation}
The first product runs over all primes and the second identity holds because $L_p+\frac{\bm{u}}{\cond} =L_p$ for $p\nmid \cond$.
By \cite[Lemma 4.1]{kane_universal_2018}, the Siegel mass and the proper Siegel mass are related by 
\begin{equation}\label{eqn:MassRelate}
	m^+\left(L+\tfrac{\bm{u}}{\cond}\right)=2m\left(L+\tfrac{\bm{u}}{\cond}\right).
\end{equation}
Using \eqref{eqn:MassRelate}, \eqref{eqn:SiegelMass} and \eqref{eqn:ProperSiegelMass} together with \eqref{eqn:massrelationZilong} imply that
\begin{align}\label{eqn:hlower}
	h_{L+\frac{\bm{u}}{\cond}} &\geq m\left(L+\tfrac{\bm{u}}{\cond}\right)=\tfrac12m^+\left(L+\tfrac{\bm{u}}{\cond}\right)=\tfrac12m^+(L)\prod_{p\mid \cond}\left[\SO(L_p):\SO\left(L_p+\tfrac{\bm{u}}{\cond}\right)\right],\\
	\nonumber 
	h_{L+\frac{\bm{u}}{\cond}}^+& \geq m^+\left(L+\tfrac{\bm{u}}{\cond}\right)=m^+(L)\prod_{p\mid \cond}\left[\SO(L_p):\SO\left(L_p+\tfrac{\bm{u}}{\cond}\right)\right].
\end{align}

\section{Associates over local rings}\label{sec:associates}

In this section, we fix an odd prime $p$ and abbreviate $v=v_p$. Let $L$ be a $\Z_p$-lattice in a $\Q_p$-vector space $V$ with associated symmetric bilinear form $B$. For $\bm{x}\in V$, define the {\it order of $\bm{x} $ in $ L $} by 
\[
	v(L;\bm{x})=v_p(L;\bm{x}):=\min_{\bm{y}\in L}v\left(\tfrac12B(\bm{x},\bm{y})\right).
\]
For $ r\in \N_0$, we also define the sublattice of $L$
\begin{equation}\label{eqn:Lrdef}
	\Lr{r} := \left\{\bm x\in L : v\left(L;\bm x\right)\ge r\right\}.
\end{equation}
Note that $L=\Lr{0}$, and hence in particular $v(\Lr{0};\bm{x})=v(L;\bm{x})$. The following properties are immediate from the definitions above. 

\begin{prop}\label{prop:ordervr}
For $r\in\N_0$ and $\bm x\in L$, the following hold.
	\begin{enumerate}[leftmargin=*,label=\textnormal{(\arabic*)}]
		\item We have 
		$
			v(\Lr{r};\bm x) \ge r.
		$
		
		\item We have, for $a\in\Z_p$, 
		\[
			v\left(\Lr{r};a\bm x\right) = v(a)+v\left(\Lr{r};\bm x\right).
		\]
		
		\item For $s\in\N_0$ with $s\geq r$, we have 
		\[
			v\left(\Lr{r};\bm x\right) \le v\left(\Lr{s};\bm x\right) \le v\left(\Lr{r};\bm x\right)+s-r.
		\]
	\end{enumerate}
\end{prop}

An element $ \bm{x}\in L $ is called {\it primitive} if $ \bm{x}\notin p L $, {\it simple of index $ r $} if $ v(L;\bm{x})=v(\Lr{r};\bm{x})=r $, and {\it orthogonal} if $ v(L;\bm{x})=v(Q(\bm{x})) $. If $ v(L;\bm{x})=v(Q(\bm{x}))=r $, then $ \bm{x} $ is simple of index $ r $.

The following proposition gives a useful equivalent condition to check if a vector is simple.
 
\begin{prop}\label{prop:simpleequiv}
	An element $\bm x\in L$ is simple of index $r\in\N_0$ if and only if there exists $\bm w\in L$ such that $v(L;\bm x)=v(\frac12B(\bm x,\bm w))=v(L;\bm w)=r$. 
\end{prop} 

\begin{proof}
	Suppose that $\bm{x} $ is simple of index $r$. Then by definition,
	\[
		v(L;\bm x) = v\left(\Lr{r};\bm x\right) = \min_{\bm y\in\Lr{r}} v\left(\tfrac12B(\bm x,\bm y)\right) = r.
	\]
	Thus there exists $ \bm{w}\in \Lr{r} $ such that $ v(\tfrac12B(\bm{x},\bm{w}))=r $. Hence 
	\begin{align*}
		r\le v(L;\bm{w})\le v\left(\tfrac12B(\bm{x},\bm{w})\right)=r=v(L;\bm{x}).
	\end{align*}
	Thus we must have equality throughout.	Conversely, suppose that there exists $\bm{w}\in L$ such that
	\[
		r=v(L;\bm{x})=v\left(\tfrac12B(\bm{x},\bm{w})\right)=v(L;\bm{w}).
	\]
	Since $v(L;\bm{w})=r$, we have $ \bm{w}\in \Lr{r} $ by definition. Combining this with Proposition \ref{prop:ordervr} (1) gives
	\[
		r\le v\left(\Lr{r};\bm{x}\right)\le v\left(\tfrac12B(\bm{x},\bm{w})\right)=r.
	\]
	Hence $ v(L;\bm{x})=v(\Lr{r};\bm{x})=r $ and thus $ \bm{x} $ is simple of index $r$.
\end{proof}
 
For $\bm x\in L$, following \cite[p.\hskip 0.1cm 908]{james_integral_1965}, we call $r\in\N_0$ a {\it critical index} of $\bm x$ if one of the following holds:
\begin{enumerate}[leftmargin=*,label=\textnormal{(\arabic*)}]
	\item We have $ r=0 $ and 
	\[
		v(L;\bm{x})<v\left(\Lr{1};\bm{x}\right).
	\]
	
	\item We have $ r>0 $ and
	\[
		v\left(\Lr{r-1};\bm{x}\right)=v\left(\Lr{r};\bm{x}\right)<v\left(\Lr{r+1};\bm{x}\right).
	\]
\end{enumerate}
Note that this definition agrees with the definition from \cite[p.\hskip 0.1cm 298]{james_associated_1968} by Proposition \ref{prop:ordervr} (3).

We next investigate the critical indices of vectors, starting with simple vectors.

\begin{lem}\label{lem:simplecritical}
	If $\bm x$ is simple of index $r$, then $r$ is the unique critical index of $\bm x$.
\end{lem}

\begin{proof}
	By Proposition \ref{prop:ordervr} (1), (3), for $s\ge r$ we have $v(\Lr{s};\bm x)=s$ and hence $r$ is the unique critical index of $\bm x$ if it is simple of index $r$.\qedhere
\end{proof}

If $L=M\perp M^\perp$, then we say that $M$ {\it splits} $L$. James \cite[Lemma 1]{james_integral_1965} proved the following.

\begin{lem}\label{lem:splitting1}
	For $\bm x\in L$, the rank-one sublattice $\Z_p\bm{x}$ of $L$ splits $L$ if and only if $v(L;\bm x)=v(Q(\bm x))$. 
\end{lem}

Since $p\ne2$, the following is well-known (for example, see \cite[Subsection 8.3]{Cassels}).

\begin{prop}\label{prop:orsplitting}
Let $p$ be an odd prime and $L$ be a $\Z_p$-lattice of rank $\ell$. Then there exists a basis of orthogonal vectors $\bm{\gamma_1},\dots,\bm{\gamma_{\ell}}\in L$ of (not necessarily distinct) indices $r_1,\dots,r_{\ell}$ such that 
\[
	L=\Z_p \bm{\gamma_{1}}\perp \cdots\perp\Z_p\bm{\gamma_{\ell}}.
\]
\end{prop}

The $r_j$ appearing in \Cref{prop:orsplitting} are also the only possible critical indices of $L$, as proven by James and Rosenzweig in \cite[Proposition 4]{james_associated_1968}.

\begin{lem}\label{lem:orsplitting}
	Let $p$ be an odd prime and $L$ be a $\Z_p$-lattice of rank $\ell$. Then the only possible critical indices of $\bm x\in L$ are the $r_j$ appearing in \Cref{prop:orsplitting}.
\end{lem}

James \cite[Proposition 5]{james_integral_1965} used Proposition \ref{prop:orsplitting} to obtain a useful decomposition of a vector based on its critical indices. Namely, he represented a vector $\bm{x}\in L$ as a linear combination of pairwise orthogonal, primitive, and simple vectors whose indices coincide with the critical indices of $\bm{x}$.

\begin{prop}\label{prop:orsimple}
	Let $\bm x\in L$ be a primitive vector whose critical indices are $r_1<\dots<r_s$ and write $v(\Lr{r_j};\bm x)=r_j+b_j$. Then $\bm x$ can be expressed as linear combination of pairwise orthogonal, primitive, and simple vectors $\bm{\l_j}\in L$ of index $r_j$, i.e.,
	\begin{align*}
		\bm{x}=\sum_{j=1}^{s}p^{b_j}\bm{\lambda_j}.
	\end{align*}
	Moreover, $b_1>\dots>b_s=0$ and $r_1+b_1<\dots<r_s+b_s$.
\end{prop}

\begin{proof}
	The statement of Proposition \ref{prop:orsimple} coincides with the statement in \cite[Proposition 5]{james_integral_1965}, except that James did not assume that $\bm x$ is primitive and does not claim that $b_s=0$. However, it is not hard to show that if $\bm x$ is primitive, then $b_s=0$.\qedhere
\end{proof}

\begin{re}\label{re:orsimple}
	Decomposing $\bm x\in L$ as in \Cref{prop:orsimple}, we have
	\[
		v(L;\bm{x}) = \min_{1\le j\le s} v\left(L;p^{b_j}\bm{\lambda_j}\right) = \min_{1\le j\le s}\left(b_j+r_j\right) = b_1+r_1 = v\left(\Lr{r_1};\bm{x}\right).
	\]
	In particular, if $\bm x$ is simple of index $r$, then $r=r_1$ and $v(L;\bm x)=v(\Lr{r_1};\bm x)=r_1$.
\end{re}

The following lemma was given by O'Meara, James, and Rosenzweig \cite[Proposition 8]{james_associated_1968}.

\begin{lem}\label{lem:asnondyadi}
	Let $\bm x,\bm y\in L$ be simple vectors in 
a $\Z_p$-lattice 
\rm
$L$ and suppose that $p\ne2$. Then $\bm x\sim_L\bm y$ if and only if $Q(\bm x)=Q(\bm y)$ and $v(L;\bm x)=v(L;\bm y)=r$.
\end{lem}

The next lemma considers whether vectors are associated if $v(Q(\bm x))=0$.

\begin{lem}\label{lem:asnux20}
	For $p\ne2$, let $L$ be a $\Z_p$-lattice and suppose that $\bm x\in L$ satisfies $v(Q(\bm x))=0$. Then we have $\bm x\sim_L\bm y$ for any $\bm y\in L$ with $Q(\bm y)=Q(\bm x)$.
\end{lem}

\begin{proof}
	Since $v(Q(\bm x))=0$, we have
	\[
		v(L;\bm{x})\le v(Q(\bm{x}))=0 
	\]
	and so $\bm x$ is simple of index zero.
 If $\bm y\in L$ satisfies $Q(\bm y)=Q(\bm x)$, then we directly obtain
	\[
		v(Q(\bm{y}))=v(Q(\bm{x}))=0.
	\]
	Hence $\bm y$ is also simple of index zero. The claim now follows by \Cref{lem:asnondyadi}.
\end{proof}

For unimodular lattices, we may use the following lemma.

\begin{lem}\label{lem:asunimodular}
	Let $L$ be a unimodular $\Z_p$-lattice with $p\ne2$ and $\bm x\in L$ be primitive with $v(Q(\bm x))>0$. Then we have $\bm x\sim_L\bm y$ for any primitive $\bm y\in L$ with $Q(\bm y)=Q(\bm x)$.
\end{lem}

\begin{proof}
	We claim that every primitive vector in $L$ is simple of index $0$. For this, we need to show that a primitive vector cannot have any critical indices $r\neq 0$.
	To see this recall that by \Cref{lem:orsplitting} the critical indices of a vector must be one of the choices of $r$ appearing in \Cref{prop:orsplitting}, so it remains to show that only $r=0$ may appear in the orthogonal splitting in \Cref{prop:orsplitting}. Note that in an orthogonal splitting $L=M_1\perp M_2$
	\begin{equation}\label{eqn:dsplit}
		d_L=d_{M_1}d_{M_2}.
	\end{equation}
	Since $L$ is unimodular, if $M_1$ and $M_2$ are integral, then they must also be unimodular, so $v(d_{M_j})=0$.
	Inductively, each block $\Z_p\bm{\g_j}$ in the orthogonal splitting from \Cref{prop:orsplitting} must also be unimodular. Since $d_{\Z_p\bm{\g_j}}=Q(\bm{\g_j})$ and \eqref{eqn:dsplit} implies that $v(d_{\Z_p\bm{g_j}})=0$, we conclude that $v(Q(\bm{\g_j}))=0$ and thus $r_j=0$ for every $1\le j\le\ell$ in \Cref{prop:orsplitting}. Thus \Cref{lem:orsplitting} implies that only critical index of $\bm z\in L$ is $0$.
	By \Cref{prop:orsimple}, each primitive $\bm z\in L$ is simple of index $0$.
	By assumption, $\bm x$ and $\bm y$ are primitive, so $v(L;\bm x)=v(L;\bm y)=0$, and, since $Q(\bm x)=Q(\bm y)$ by assumption, we conclude that $\bm x\sim_L\bm y$ by \Cref{lem:asnondyadi}.
\end{proof}

\section{The ratio of the masses between lattices and shifted lattices}\label{sec:mass}

Let $L$ be a $\Z$-lattice in a quadratic space $(V,Q)$ over $\Q$. Suppose that $\bm u\in L$ is primitive and $\cond\in\N$. The goal of this section is to obtain a lower bound for the Siegel mass $m^+(L+\frac{\bm u}{\cond})$ (and by \eqref{eqn:MassRelate} also for $m(L+\frac{\bm{u}}{\cond})$). As noted in \cite[Remark 6]{shimura_anexact_1999}, \eqref{eqn:massrelationZilong} yields a formula for $m^+(L+\frac{\bm{u}}{\cond})$ if one computes $m^+(L)$ and finitely many local factors $ [\SO(L_p):\SO(L_p+\frac{\bm{u}}{\cond})]$. For a prime $p\mid\cond$, we write $\cond=p^k\condd$ for some $k\in\N$ and $p\nmid\condd$. Since $\condd$ is invertible in $\Z_p$, we have\footnote{We let $\condd^{-1}$ denote the inverse of $d$ in $\Z_p$ under the natural embedding of $\Q$ into $\Q_p$.} 
\[
	\bm{u_p}:=\condd^{-1} \bm u\in L_p, \quad\tfrac{\bm u}{\cond}=\tfrac{\bm{u_p}}{p^k}.
\]
Hence it suffices to bound $m^+(L)$ and 
\begin{equation}\label{eqn:indexppower}
	\left[\SO(L_p) : \SO\left(L_p+\tfrac{\bm{u}}{\cond}\right)\right]=\left[\SO(L_p) : \SO\left(L_p+\tfrac{\bm{u_p}}{p^k}\right)\right]
\end{equation}
from below for every $p\mid\cond$ in order to obtain a lower bound for $m^+(L+\frac{\bm u}{\cond})$.
The primary property that we use to obtain a lower bound for \eqref{eqn:indexppower} is the fact that it is primitive, so for ease of notation we simply write $\bm u$ instead of $\bm{u_p}$ in our calculations below. To estimate \eqref{eqn:indexppower}, we bound
\[
	\left[O(L_p) : O\left(L_p+\tfrac{\bm{u}}{p^k}\right)\right]
\]
and then use \eqref{eqn:OSOinequality}. For each element $\ol\s \in O(L_p)/O(L_p+\frac{\bm u}{p^k})$, we obtain an element $\ol{\s(\bm u)}\in L_p/p^kL_p$, and one easily checks that these are distinct for different choices of $\ol\s$. 
We therefore conclude that
\begin{equation}\label{eq:localfactorexpression}
	\left[O(L_p) : O\left(L_p+\tfrac{\bm{u}}{p^k}\right)\right]= \#\left\{\ol{\s(\bm{u})}\in L_p/p^kL_p : \s\in O(L_p)\right\}.
\end{equation}	 
Let $\bm u$ be a primitive vector of $L$ with $Q(\bm u)=m\in\N_0$. For $k,r\in\N_0$, we define 
\begin{align}\label{eqn:Rupk}
	R\left(Q;\bm u,p^k\right) &:= \left\{\bm x\in L_p/p^kL_p : \bm x\notin pL_p,Q(\bm x)\equiv Q(\bm u)\pmod{p^k}\right\},\\
	\label{eqn:Rurpk}
	R\left(Q;\bm u,r,p^k\right) &:= \left\{\bm x\in R\left(Q;\bm u,p^k\right) : \bm x\equiv\bm u\pmod{p^r}\right\},\\
	r\left(Q;\bm{u},p^k\right) &:= \# R\left(Q;\bm{u},p^k\right),\quad r\left(Q;\bm{u},r,p^k\right) := \# R\left(Q;\bm{u},r,p^k\right).\nonumber
\end{align}
If $ r=0 $, then we have $ R(Q;\bm{u},0,p^{k})=R(Q;\bm{u},p^{k}) $, 
and for $1\le k\le r$ we have
\[
	R\left(Q;\bm{u},r,p^{k}\right)=\left\{\bm{u}+p^kL_p\right\}.
\]
For $ \bm{x},\bm{y}\in L_{p} $, we define the equivalence $ \sim_{r} $ on $ R(Q;\bm{u},p^{k}) $ by $ \bm{x}\sim_{r} \bm{y} $ if $ \bm{x}\equiv \bm{y}\Pmod{p^{r}}$, i.e., $ \bm{x}-\bm{y}\in p^{r}L_{p} $.

\begin{lem}\label{lem:QxQgammau}
Let $L$ be an integral $\Z$-lattice, $p$ an odd prime, $k\in\N$, and $r\in\N_0$. Suppose that $\bm u\in L_p$ satisfies $v_p(Q(\bm u))=j\in\N_0$. For any $\bm x\in R(Q;\bm u,p^{k+j})$ there exists $\g\in\Z_p^\times$ with $\g\equiv1\Pmod{p^k}$ and $Q(\bm x)=Q(\g\bm u)$.
\end{lem}

\begin{proof}
	Let $\bm x\in L_p$ satisfy $Q(\bm x)\equiv Q(\bm u)\Pmod{p^{k+j}}$. We may choose $\a\in\Z_p$ such that
	\[
		Q(\bm{x})=Q(\bm{u})+p^{k+j}\alpha.
	\]
	Since $v_p(Q(\bm u))=j$, $\frac{Q(\bm u)}{p^j}\in\Z_p^\times$, by Hensel's Lemma, there exists $\b\in\Z_p$ such that
	\begin{align*}
		1+p^{k}\left(\tfrac{Q(\bm{u})}{p^j}\right)^{-1}\alpha = \left(1+\beta p\right)^2.
	\end{align*}
	Therefore 
	\[
		Q(\bm x) = Q(\bm u) + p^{k+j}\a = \left(1+\b p \right)^2Q(\bm u).
	\]
	Setting $\g:=1+\b p$, we thus have $Q(\bm x)=Q(\g\bm u)$. Since
	\begin{equation*}
		Q(\gamma \bm{u}) = Q\left(\bm{u}+\beta p\bm{u}\right) = Q(\bm{u})+B\left(\bm{u},\beta p \bm{u}\right)+\beta^{2} p^{2}Q(\bm{u})
		= Q(\bm{u})+2\beta p Q( \bm{u})+\beta^{2} p^{2}Q(\bm{u}),
	\end{equation*}
	we have
	\[
		\b p^{j+1}\left(2\tfrac{Q(\bm u)}{p^{j}}+\b p\tfrac{Q(\bm u)}{p^j}\right) = 2\b pQ(\bm u) + \b^2p^2Q(\bm{u}) = Q(\g\bm u) - Q(\bm u) = Q(\bm x) - Q(\bm u) \in p^{k+j}\Z_p.
	\]
	Since $ \frac{Q(\bm{u})}{p^{j}}\in \Z_{p}^{\times} $ and $p\neq 2$, $2\frac{Q(\bm u)}{p^j}\in\Z_p^\times$ and it follows that $\b\in p^{k-1}\Z_p$.
	Therefore, we have 
	\[
		\g = 1 + \b p \equiv 1\pmod{p^{k}}. \qedhere
	\]
\end{proof}

We use the following lemma to obtain a lower bound for $[O(L_p):O(L_p+\frac{\bm{u}}{p^k})]$ if $p\nmid \gcd(d_L,Q(\bm{u}))$.

\begin{lem}\label{lem:index_unit_unimodular}
	Let $L$ be an integral $\Z$-lattice, $p$ an odd prime, $k\in\N$, and $j,r\in\N_0$. Let $\bm u\in L_p$ be primitive satisfying $Q(\bm u)\in p^j\Z_p^\times$. If either $j=0$ or ($L_p$ is unimodular and $j\in\N$), then
	\[
		\left[O(L_p) : O\left(L_p+\tfrac{\bm u}{p^k}\right)\right] \ge \left|R\left(Q;\bm u,p^{k+j}\right)\Big/\sim_k\right|.
	\]
	In particular,
	\begin{align*}
		\left[O(L_{p}):O\left(L_{p}+\tfrac{\bm{u}}{p^{k}}\right)\right]\ge r\left(Q;\bm{u},p^{k}\right).
	\end{align*}
\end{lem}

\begin{proof}
	By \Cref{lem:QxQgammau}, for $\bm x\in R(Q;\bm u,p^{k+j})$, there exists $\g\in\Z_p^\times$ with $\g\equiv1\Pmod{p^k}$ such that $Q(\bm x)=Q(\g\bm u)$. Since $Q(\bm x)=Q(\g\bm u)$, if $j=0$, then $\bm x\sim_L \g\bm u$ by Lemma \ref{lem:asnux20}. If $j\in\N$ and $L_p$ is unimodular, then $\bm x\sim_L \g\bm u$ by Lemma \ref{lem:asunimodular}. In either case, there exists $\s_{\bm{x}}\in O(L_p)$ such that $\s_{\bm{x}}(\g\bm u)=\bm x$ and so $\s_{\bm{x}}(\bm u)\equiv\bm x\Pmod{p^k}$.
	
	It is not hard to check that $\ol{\sigma_{\bm{x}}(\bm{u})}$ is well-defined for $\ol{\bm{x}}\in R(Q;\bm u,p^{k+j})/\sim_k$ and different elements of $R(Q;\bm u,p^{k+j})/\sim_k$ give distinct elements of $\{\ol{\sigma(\bm u)}\in L_p/p^kL_p:\sigma\in O(L_p)\}$.
	Therefore
	\begin{align*}
		\left|R\left(Q;\bm u,p^{k+j}\right)\Big/\sim_k\right|&=\left|\left\{\ol{\s_{\bm{x}}(\bm u)}\in L_p/p^kL_p : \ol{\bm{x}}\in R\left(Q;\bm u,p^{k+j}\right)\Big/\sim_k\right\}\right|\\
		&\le \left|\left\{\ol{\s(\bm u)}\in L_p/p^kL_p : \s\in O(L_p)\right\}\right| 
	\end{align*}
	Combining with \eqref{eq:localfactorexpression} implies that 
	\[
		\left|R\left(Q;\bm u,p^{k+j}\right)\Big/\sim_k\right| \le \left|\left\{\ol{\s(\bm u)}\in L_p/p^kL_p : \s\in O(L_p)\right\}\right| = \left[O(L_p) : O\left(L_p+\tfrac{\bm u}{p^k}\right)\right]. \qedhere
	\]
\end{proof}

For $\cond\in\N$, we set  $\Omega = \Omega_\cond := \{p : p\mid\cond\}$. For a shifted lattice $L+\frac{\bm u}{\cond}$ of conductor $\cond$, we split the primes appearing in $\Omega$ via
\begin{align*}
	\Omega_{1,\cond}(L,\bm{u}) = \Omega_1\left(L+\tfrac{\bm u}{\cond}\right) &:= \left\{p\in\Omega_\cond : Q(\bm u)\in \Z_p^\times \text{ or ($L_p$ is unimodular and } Q(\bm{u})\in p\Z_p)\right\},\\
	\Omega_{2,\cond}(L,\bm{u}) = \Omega_2\left(L+\tfrac{\bm u}{\cond}\right) &:= \left\{p\in\Omega_\cond : L_p \text{ is not unimodular and } Q(\bm u)\in p\Z_p)\right\}.
\end{align*}
Note that since $L_p$ is unimodular if and only if $p\nmid d_L$, we have $p\in \Omega_{2,\cond}(L,\bm{u})$ if and only if $p\mid \gcd(d_L,Q(\bm{u}),\cond)$. We assume throughout that $\Omega_{2,\cond}(L;\bm{u})=\emptyset$, so that 
\begin{equation}\label{eqn:Omegacgoodsplit}
\Omega_{\cond}=\Omega_{1,\cond}(L;\bm{u}).
\end{equation}
Set $\cond_{p}:=\ord_{p}(\cond)\ge 1$. We often fix the lattice $L$ and vary $\bm{u}$, omitting the dependence on $L$ in the notation in such cases.

\begin{thm}\label{thm:generalclassformula}
	Let $L$ be an integral $\Z$-lattice on a quadratic space $(V,Q)$ over $\Q$ and let $\bm u\in L$ be a primitive vector with $Q(\bm u)\ne0$. Let $\cond$ be a positive odd integer for which $\Omega_{2,\cond}(L;\bm{u})=\emptyset$. Then
	\[
		\frac{m^+\left(L+\frac{\bm u}{\cond}\right)}{m^+(L)} = \frac{2m\left(L+\frac{\bm u}{\cond}\right)}{m^+(L)} \ge \prod_{p\mid \cond} \frac{r\left(Q;\bm u,p^{\cond_p}\right)}{2}.
	\]
\end{thm}

\begin{proof}
	We first combine \eqref{eqn:MassRelate}, \eqref{eqn:massrelationZilong}, \eqref{eqn:indexppower}, and then \eqref{eqn:OSOinequality} to bound
	\begin{align*}\nonumber
		\frac{2m\left(L+\frac{\bm u}{\cond}\right)}{m^+(L)} &{=} \frac{m^+\left(L+\frac{\bm u}{\cond}\right)}{m^+(L)} {=} \prod_{p\mid\cond} \left[\SO(L_p) : \SO\left(L_p+\tfrac{\bm u}{\cond}\right)\right]
		= \prod_{p\mid\cond} \left[\SO(L_p) : \SO\left(L_p+\tfrac{\bm{u_p}}{p^{\cond_P}}\right)\right]\\
		&{\ge} \prod_{p\mid\cond} \tfrac12\left[O(L_p) : O\left(L_p+\tfrac{\bm{u_p}}{p^{\cond_p}}\right)\right].
	\end{align*}
	By assumption, \eqref{eqn:Omegacgoodsplit} holds. Hence for $p\mid \cond$ we have $p\in\Om_{1,\cond}(\bm u)$, so Lemma \ref{lem:index_unit_unimodular} yields
	\[
 		\left[O(L_p):O\left(L_p+\tfrac{\bm{u_p}}{p^{\cond_p}}\right)\right] \ge r\left(Q;\bm{u_p},p^{\cond_p}\right).\qedhere
	\]
\end{proof}

\section{Proof of Theorem \ref{thm:classbound}, Corollary \ref{cor:ClassNumberOne}, and Theorem \ref{thm:diagonalclassformulaintro}}\label{sec:proofmainresult}

By Theorem \ref{thm:generalclassformula}, the ratio of the masses for $ L+\frac{\bm{u}}{\cond} $ and $ L $ is bounded below by the number of solutions of some quadratic congruence in $ \Z_{p} $ for which the conductor is divisible by $ p $. In general, it is difficult to determine the number of solutions in a local ring $ R_{p} $, but in the ring of $ p $-adic integers, for $m\in\Z$ and $Q$ a diagonal quadratic form, Li and Ouyang \cite{li_counting_2018} deduced some explicit formulas for counting the number of solutions to the quadratic congruence $Q(\bm{x})\equiv m\Pmod{p^k}$ with $\bm{x}$ restricted to certain subsets of $\Z^{\ell}/p^{k}\Z^{\ell}$. In this section, based on \cite{li_counting_2018}, we study the number of solutions of such quadratic congruences with additional restrictions and then prove our \Cref{thm:classbound} by applying Theorem \ref{thm:generalclassformula}.

We assume that $ L=\mathbb{Z}^{\ell} $ is associated with a diagonal quadratic form in $ \mathbb{Z}[x_{1},\cdots,x_{\ell}] $, 
\begin{align*}
	Q(\bm{x})=Q\left(x_{1},\ldots,x_{\ell}\right):=a_{1}x_{1}^{2}+\cdots+a_{\ell}x_{\ell}^{2}.
\end{align*}
Let $\bm u\in\Z^\ell$ be a primitive vector (which in this case means $\gcd(u_1,u_2,\dots,u_{\ell})=1$) with $m:=Q(\bm u)\ne 0$ and $I=\{1,\dots,\ell\}$. Define $I_{\bm u,p}:=\{j\in I:p\nmid u_j\}$. The {\it depth} of $Q$ at an odd prime $p$ associated to $\bm u$ is
\begin{equation}\label{eqn:depthdef}
	d_{\bm{u},p}=d_{\bm{u},p}(Q):=\min_{j\in I_{\bm{u},p}}\left(\ord_{p}(a_j)+1\right).
\end{equation}

For $S\subseteq I$, we denote
\[
	R_S\left(Q;\bm u,p^k\right) := \left\{\bm x\in\Z^\ell/p^k\Z^\ell : Q(\bm x)\equiv Q(\bm u)\pmod{p^k}\text{ and }p\nmid x_j\text{ for all }j\in S\right\} 
\]
and set (also see the corresponding notation $ N_{J}(Q;c,n) $ in \cite[p.\hskip 0.1cm 2]{li_counting_2018})
\[
	r_{S}\left(Q;\bm{u},p^{k}\right):=\# R_{S}\left(Q;\bm{u},p^{k}\right).
\]

Suppose that $p\nmid m$ or ($p\mid m$ and $p\nmid\prod_{d=1}^\ell a_d$). In the former case, $Q(\bm u)\equiv m\Pmod{p^k}$ implies that there exists some $j\in I$ such that $p\nmid a_ju_j^2$, and hence neither $a_j$ nor $u_j$ is divisible by $p$. In the latter case, since $\bm u$ is primitive, there exists some $j$ such that $p\nmid u_j$ and $p\nmid\prod_{d=1}^\ell a_d$ implies that $p\nmid a_j$ in particular. Combining, there must exist some $j=j_p$ such that $p\nmid a_j$ and $p\nmid u_j$ in both cases. Note that we have
\begin{equation}\label{eqn:rJ0lower}
	r\left(Q;\bm{u},p^k\right) \ge r_{\{j_p\}}\left(Q;\bm{u},p^k\right) \ne 0.
\end{equation}

For $p\neq 2$, we let
\[
	I_{p}:=I_{\bm{a},p}=\{j\in I:p\nmid a_j\}
\]
and set $\De_p:=(\frac{-1}{p})p$, where $(\frac{-1}{p})$ is the Legendre symbol. Let $t_p:=\# I_p$ and for $J\subseteq I$ set
\[
	s_{J,p}=s_{\bm{a},J,p}:=\#\{j\in J: p\nmid a_j\}\quad\text{ and }\quad r_{J,p}:=\# \left\{j\in J : \left(\tfrac{a_j}{p}\right)=-1\right\}.
\]
In particular, we note that $s_{\{j_p\},p}=1$ by construction and we set $r_p:=r_{\{j_p\},p}$.
The evaluation of $r_{\{j_p}(Q;\bm u;p^k)$ from \cite{li_counting_2018} splits into a number of cases, so for ease of notation we define
\begin{equation}\label{eqn:Cpmtprp}
	 C_{p}\left(m,t_{p},r_{p}\right):=
	 \begin{cases}
	 	\dfrac{1+p\left(\frac mp\right)}{p-1}&\text{if $p\nmid m$, $2\nmid t_p$, and $r_{p}=0$,}\\
	 	\dfrac{1-p\left(\frac mp\right)}{p-1}&\text{if $p\nmid m$, $2\nmid t_p$, and $r_{p}=1$,}\\
	 	\dfrac{-1-\left(\frac {-m}p\right)}{p-1}&\text{if $p\nmid m$, $2\mid t_p$, and $r_{p}=0 $,}\\
	 	\dfrac{1-\left(\frac {-m}p\right)}{p-1} &\text{if $p\nmid m$, $2\mid t_p$, and $r_{p}=1$,}\\
	 	-1 &\text{if $p\mid m$ and $2\nmid t_p$,}\\
	 	1&\text{if $p\mid m$, $2\mid t_p$, and $r_{p}=0$.}\\
	 	-1&\text{if $p\mid m$, $2\mid t_p$, \and $r_p=1$.}
	 \end{cases}
\end{equation}
We next combine \eqref{eqn:rJ0lower} with \cite[Theorem 4.4 (2)]{li_counting_2018} (taking $J=\{j_p\}$, $t=\ell$, and $a=k$).
\begin{lem}\label{lem:exactformula_modulooddprime}
	Let $k\in\N$ and $ p $ an odd prime. Suppose that $ p\nmid m $ or ($ p\mid m $ and $ p\nmid \prod_{j=1}^{\ell}a_j $).
Then 
	\begin{align*}
		r\left(Q;\bm{u},p^k\right) \ge r_{\{j_p\}}\left(Q;\bm{u},p^k\right)= p^{(\ell-1)(k-1)}p^{\ell}\left(1-\tfrac1p\right)\left(\tfrac1p+\tfrac{\Delta_p^{\left\lfloor\frac{t_p}{2}\right\rfloor}}{p^{t_p}} C_p(m,t_p,r_p)\right).
	\end{align*}
\end{lem}

We next bound $\frac{\Delta_p^{\lfloor\frac{t_p}{2}\rfloor}}{p^{t_p}} C_{p}\left(m,t_p,r_p\right)$ to obtain explicit inequalities for $r(Q;\bm{u},p^k)$ from Lemma \ref{lem:exactformula_modulooddprime}.

\begin{lem}\label{lem:formula_modulooddprime}
	Let $k\in\N$ and $p$ an odd prime. Suppose that $p\nmid m$ or ($p\mid m$ and $t_p\geq 3$).
	\begin{enumerate}
		\item If $p\nmid m$, then
		\[
			r\left(Q;\bm{u},p^k\right) \ge p^{(\ell-1)k} \left(1-\tfrac1p\right)
			\begin{cases}
				\frac13 & \text{if }p=3,\\				1-\frac{2}{p-1} & \text{if }p\ge5.\\
			\end{cases}
		\] 
		
		\item If $p\mid m$ and $t_p\geq 3$, then 
		\[
			r\left(Q;\bm{u},p^k\right) \ge p^{(\ell-1)k}\left(1-\tfrac1p\right)^2.
		\]
	\end{enumerate}
\end{lem}

\begin{proof}
	We first state a few properties which are useful throughout the proof. First note that
	\begin{equation}\label{eqn:CpBound}
		-1\leq C_p\left(m,t_p,r_p\right)\leq \tfrac{p+1}{p-1}.
	\end{equation}
	Moreover, by \eqref{eqn:rJ0lower} we have $r_{\{j_p\}}(Q;\bm{u},p^k)\neq 0$, and hence the equality in Lemma \ref{lem:exactformula_modulooddprime} implies that 
	\begin{equation}\label{eqn:Cpnonvanish}
		\tfrac{\Delta_p^{\left\lfloor \frac{t_p}{2}\right\rfloor}}{p^{t_p}}C_p(m,t_p,r_p)\neq -\tfrac1p.
	\end{equation}
	(1) Suppose that $p\nmid m$. By Lemma \ref{lem:exactformula_modulooddprime}, we have 
	\[
		r\left(Q;\bm{u},p^k\right) \geq p^{(\ell-1)k}\left(1-\tfrac{1}{p}\right)\left(1+p\tfrac{\Delta_p^{\left\lfloor\frac{t_p}{2}\right\rfloor}}{p^{t_p}} C_{p}\left(m,t_p,r_p\right)\right).
	\]
	To obtain the claim, we hence need to show that 
	\begin{equation}\label{eqn:DeltaCpBound}
		\frac{\Delta_p^{\left\lfloor\frac{t_p}{2}\right\rfloor}}{p^{t_p}} C_{p}\left(m,t_p,r_p\right)\geq
		\begin{cases}
 			-\frac{2}{9}&\text{if }p=3,\\
			-\frac{2}{p(p-1)}&\text{if }p\geq 5.
		\end{cases}
	\end{equation}
	If $t_p=1$, then \eqref{eqn:Cpnonvanish} implies that $C_p(m,t_p,r_p)\neq -1$, from which we conclude that
	\[
		C_p(m,t_p,r_p)=\tfrac{p+1}{p-1}>0.
	\]
	If $t_p\geq3$ is odd, then \eqref{eqn:CpBound} implies that
	\[
		\frac{\Delta_p^{\left\lfloor \frac{t_p}{2}\right\rfloor}}{p^{t_p}}C_p(m,t_p,r_p)\geq -\tfrac{1}{p^2} \tfrac{p+1}{p-1}\geq 
		\begin{cases}
			-\frac{2}{9}&\text{if }p=3,\\
			-\frac{2}{p(p-1)}&\text{if }p\geq 5.\\
		\end{cases}
	\]
	We therefore see that \eqref{eqn:DeltaCpBound} holds for $t_p$ odd. 

	On the other hand, if $t_p$ is even, then \eqref{eqn:CpBound} (combined with \eqref{eqn:Cpnonvanish} for $p=3$) implies that
	\[
		\frac{\Delta_p^{\left\lfloor \frac{t_p}{2}\right\rfloor}}{p^{t_p}}C_p(m,t_p,r_p)\geq 
		\begin{cases}
			-\frac{1}{9}&\text{if }p=3,\\
 			-\frac{2}{p(p-1)}&\text{if }p\geq 5.
		\end{cases}
	\]
	We therefore conclude \eqref{eqn:DeltaCpBound}, completing the proof of part (1).\\
	(2) Suppose that $p\mid m$ and $t_p\geq 3$. Again using \Cref{lem:exactformula_modulooddprime}, we have to show that 
	\[
		\frac{\Delta_p^{\left\lfloor\frac{t_p}{2}\right\rfloor}}{p^{t_p}} C_{p}\left(m,t_p,r_p\right)\geq -\tfrac{1}{p^2}.
	\]
 	Since $t_p\geq 3$, this follows immediately.
\end{proof}

We next bound $r(Q;\bm{u},r,p^k)$ for $k\geq \max\{d_{\bm{u},p},r\}$.

\begin{lem}\label{lem:formula_notunimodular_notunit}
	Let $k,r\in\N$ and $p$ a prime. If $ k\ge\max\{d_{\bm u,p},r\}$, then
	\[
		r\left(Q;\bm{u},r,p^{k}\right)\ge p^{(\ell-1)\left(k-\max\{d_{\bm{u},p},r\}\right)}.
	\]
\end{lem}

\begin{proof}
	For $k>d_{\bm u,p} $ and $S\subseteq I$, it was shown in \cite[proof of Theorem B (1)]{li_counting_2018} that for every $\bm v\in R_S(Q;\bm u,p^{k-1})$ there exist precisely $p^{\ell-1}$ elements $\bm{v_1},\dots,\bm{v_{p^{\ell-1}}}$ of $R_S(Q;\bm u,p^k) $ with $\bm{v_j}\equiv\bm v\Pmod{p^{k-1}}$.
	Suppose that $k\geq d \geq d_{\bm{u},p} $. Since $\bm{u}\in R_{I_{\bm{u},p}}(Q;\bm{u},d,p^{k})$, applying \cite[the proof of Theorem B (1)]{li_counting_2018} inductively, we conclude that 
\rm
	\[
		r\left(Q;\bm{u},d,p^{k}\right) \ge r_{I_{\bm{u},p}}\left(Q;\bm{u},d,p^{k}\right) \ge p^{(k-d)(\ell-1)}.
	\]
	Taking $d=\max\{d_{\bm u,p},r\}$, we conclude that for $k\ge \max\{d_{\bm u,p},r\}$ we have 
	\[
		r\left(Q;\bm u,r,p^k\right) \ge r\left(Q;\bm u,\max\{d_{\bm u,p},r\},p^k\right) \ge p^{\left(k-\max\{d_{\bm u,p},r\}\right)(\ell-1)}. \qedhere
	\]
\end{proof}

We next combine Lemmas \ref{lem:formula_modulooddprime} and \ref{lem:formula_notunimodular_notunit} with Theorem \ref{thm:generalclassformula} to obtain an explicit lower bound for $\frac{m(L+\frac{\bm u}{\cond})}{m(L)}$. To state the result, for an odd prime $p$ we define
\begin{equation}\label{eqn:fpdef}
	f_p =f_{p}\left(\cond_p\right):=
	\begin{cases}
		\frac19 &\text{if } p=3,\vspace{0.2cm}\\
		\frac12-\frac{3}{2p}&\text{if } p\ge5.
	\end{cases}
\end{equation}

\begin{thm}\label{thm:diagonalclassformula}
	Let $L$ be a $\Z$-lattice on an $\ell$-dimensional ($\ell\ge3$) vector space $V$ over $\Q$ associated with a positive-definite diagonal quadratic form $Q$. Let $\bm u\in L$ be a primitive vector with $Q(\bm u)\ne0$ and $\cond\in\N$ be odd such that $\gcd(d_L,Q(\bm{u}),\cond)=1$. Then we have
	\[
		\dfrac{m\left(L+\frac{\bm{u}}{\cond}\right)}{m(L)}\ge\cond^{\ell-1}\prod_{p\mid \cond}f_{p}.
	\]
	In particular, for $ \delta>0 $, there exists a constant $ C_{\delta}>0 $ depending on $ \delta $ such that 
	\begin{align*}
		\frac{m\left(L+\frac{\bm{u}}{\cond}\right)}{m(L)}\ge C_{\delta}\dfrac{\cond^{\ell-1-\delta}}{D_L^{2\ell-2}}.
	\end{align*}
\end{thm}
\begin{proof}
Since $\gcd(d_L,Q(\bm{u}),\cond)=1$, \eqref{eqn:Omegacgoodsplit} holds. Let $p\mid \cond$. If $p=3$, then, we have, by \Cref{lem:formula_modulooddprime},
	\[
		r(Q;\bm{u},3^{\cond_{3}})\geq 2\tfrac{3^{(\ell-1)\cond_{3}}}{9}.
	\]
	If $ p\ge 5$, then $1-\frac1p\ge1-\frac{2}{p-1}$ and hence Lemma \ref{lem:formula_modulooddprime} implies that
	\[
		r\left(Q;\bm u,p^{\cond_p}\right) \ge p^{(\ell-1)\cond_p}\left(1-\tfrac1p\right)\left(1-\tfrac{2}{p-1}\right) = p^{(\ell-1)\cond_p}\tfrac{p-3}{p}.
	\]
	Combining these, we have
	\begin{equation}\label{eq:part1r}
		\prod_{p\in\Omega_{1,\cond}(\bm{u})} \frac{r\left(Q;\bm{u},p^{\cond_p}\right)}{2} \ge \prod_{p\in\Omega_{1,\cond}(\bm{u})} p^{(\ell-1)\cond_p}f_p.
	\end{equation}
	Plugging \eqref{eq:part1r} and \eqref{eqn:MassRelate} into Theorem \ref{thm:generalclassformula} thus yields
 	\begin{equation}\label{eqn:mainfirstclaim}
 		\frac{m\left(L+\frac{\bm{u}}{\cond}\right)}{m(L)} \ge \prod_{p\mid \cond} \frac{r\left(Q;\bm{u},p^{\cond_p}\right)}{2} \ge \prod_{p\mid \cond}p^{(\ell-1)\cond_p}f_p = \cond^{\ell-1} \prod_{p\mid \cond} f_p.
 	\end{equation}
This is the first claim.

For the second assertion, we claim that for $\d>0$ there exists a constant $C_{\delta}$ such that
	\begin{equation}\label{eq:prodfp}
		\prod_{p\mid \cond}f_p \geq \frac{C_\d}{\cond^\d}.
	\end{equation}
	Set $S_{\delta,p}:=\{r\in\N: p^{r\delta} f_p <1\}$. Then, since $p^{r\delta}$ is increasing in $r$, we have
	\[
		\cond^\delta \prod_{p\mid\cond} f_p = \prod_{p^r\|\cond} p^{r\delta} f_p \ge \prod_{\substack{p\text{ prime}\\S_{\delta,p}\ne\emptyset}} \min_{r\in S_{\delta,p}} p^{r\delta} f_p = \prod_{\substack{p\text{ prime}\\S_{\delta,p}\ne\emptyset}} p^\delta f_p.
	\]

	Let
	\begin{equation}\label{eqn:Cdeltadef}
		C_{\delta}:=\prod_{\substack{p\text{ prime}\\ S_{\delta,p}\neq \emptyset}} p^{\delta} f_p
	\end{equation}
	and note that the product is finite. Note that $C_\d$ only depends on $\d$ and \eqref{eq:prodfp} is proved. 
	
	Plugging \eqref{eq:prodfp} into \eqref{eqn:mainfirstclaim} yields 
	\[
		\frac{m\left(L+\frac{\bm{u}}{\cond}\right)}{m(L)} \ge \cond^{\ell-1} \prod_{p\mid\cond} f_p \ge C_\delta \cond^{\ell-1-\delta}.\qedhere
	\]
\end{proof}

We are now ready to prove \Cref{thm:diagonalclassformulaintro}.

\begin{proof}[Proof of Theorem \ref{thm:diagonalclassformulaintro}]
	The theorem follows directly from \Cref{thm:diagonalclassformula}.
\end{proof}

We next conclude Theorem \ref{thm:classbound}.

\begin{proof}[Proof of Theorem \ref{thm:classbound}]
By the first inequality in \eqref{eqn:hlower}, we have
$h_{L+\bm{\frac u\cond}} \ge m(L+\frac{\bm u}{\cond})$, so the result follows directly from Theorem \ref{thm:diagonalclassformulaintro}.
\end{proof}

\noindent We are now ready to prove an explicit version of \Cref{cor:ClassNumberOne}, giving the bounds $\cond_{\ell}$ in \Cref{tab:ClassNumberOne}.
\begin{longtable}{|c||c|c|c|c|c|c|c|c|}
\hline
$\ell$& 3&4&5&6&7&8&9&10\\
\hline
$c_{\ell}$&81&49&23&23&23&33&23&19\\
\hline
\caption{Bounds $\cond\leq \cond_{\ell}$ for conductors of class number one shifted lattices of rank $3\leq \ell\leq 10$.}\label{tab:ClassNumberOne}
\end{longtable}

\begin{proof}[Proof of \Cref{cor:ClassNumberOne}]
By \eqref{eqn:hlower} and Theorem \ref{thm:diagonalclassformulaintro}, we have
	\begin{equation}\label{eqn:hlowereval}
		h_{L+\frac{\bm{u}}{\cond}}\geq m\left(L+\frac{{u}}{\cond}\right) \ge m(L) C_{\delta} \cond^{\ell-1-\delta}.
	\end{equation}
	Since $h_L\le h_{L+\frac u\cond}$ by \cite[Corollary 4.4]{chan_representations_2013}, if $h_{L+\frac u\cond}=1$, then $h_L=1$, so we may restrict to those $L$ with $h_L=1$ classified by Watson \cite{WatsonSmallClassNumber,Watson1,Watson3,Watson2} (see Lorch and Kirschmer's tables in \cite[Subsection 7.1]{KL1} for a complete and explicit enumeration).\\
	(1) Consider diagonal ternary quadratic lattices (i.e., $\ell=3$) $L$ with class number one. Note that, for all $L$ with $h_L=1$, 
	\[
		\frac{1}{m(L)} = |O(L)| \le 48 \quad\text{ and }\quad d_L \le 1728.
	\]
	This follows from Watson's classification \cite{Watson3} of ternary quadratic forms with class number one.
	Plugging this into \eqref{eqn:hlowereval}, if $h_{L+\frac{\bm{u}}{\cond}}=1$, then we conclude that
	\begin{equation}\label{eqn:condbound}
		\cond^{2-\d} \leq \frac{1}{C_{\d}m(L)}\leq \frac{48}{C_{\d}}.
	\end{equation}
	It remains to show a lower bound for $C_\d$. Choosing $\d=\frac12$, we compute that $S_{\frac12,p}=\emptyset$ for $p>7$ and
	\begin{equation}\label{eqn:C1/2bound}
		C_{\frac{1}{2}}\geq 0.065.
	\end{equation}
	Plugging into \eqref{eqn:condbound}, we have
	\[
		\cond^{\frac{3}{2}}\leq \frac{48}{0.065},
	\]
	which implies that $\cond\le 81$ (note that $\cond$ is odd).\\
	(2) The full lists of lattices of class number one for $4\le\ell\le10$ were stated in \cite[Section 7]{KL1}. Moreover, the determinant and $\frac{1}{m(L)}$ were given for each class number one lattice, which are precisely the quantities that we require to obtain an upper bound on $c$. In the case $\ell=4$, \cite[Table 2]{KL1} and the table in \cite[Subsubsection 7.1.7]{KL1} yield that if $L$ has class number one, then
	\[
		\frac{1}{m(L)}\leq 1152\quad\text{ and }\quad d_L\leq 2^{4}\cdot3^{3}\cdot11^3.
	\]
	Thus if $h_{L+\frac{\bm u}{\cond}}=1$, then \eqref{eqn:hlowereval} implies that
	\[
		\cond^{3-\d} \le \frac{1}{C_\d m(L)} \le \frac{1152}{C_\d}.
	\]
	Choosing $\d=\frac12$ and plugging in \eqref{eqn:C1/2bound} thus yields $\cond^\frac52\le17723$. This implies that $\cond\le49$.\\
	(3) Suppose that $\ell=5$. In this case, \cite[Table 2]{KL1} and the table in \cite[Subsubsection 7.1.6]{KL1} yield that if $L$ has class number one, then 
	\[
		\frac{1}{m(L)}\leq 3840\quad\text{ and }\quad d_L\leq 2^{12} \cdot 7^{4}.
	\]
	Thus if $h_{L+\frac{\bm{u}}{\cond}}=1$, then \eqref{eqn:hlowereval} implies that
	\[
		\cond^{4-\d} \le \frac{d_L^8}{C_\d m(L)} \le \frac{3840}{C_\d}.
	\]
	Choosing $\d=\frac12$ and plugging in \eqref{eqn:C1/2bound} thus yields $\cond^{\frac{7}{2}}\leq 59077$, which gives that $\cond\leq 23$.\\
	(4) Suppose that $\ell=6$. In this case, \cite[Table 2]{KL1} and the table in \cite[Subsubsection 7.1.5]{KL1} yield that if $L$ has class number one, then 
	\[
		\frac{1}{m(L)}\leq 103680\quad\text{ and }\quad d_L\leq 23^5.
	\]
	Thus if $h_{L+\frac{\bm{u}}{\cond}}=1$, then \eqref{eqn:hlowereval} implies that
	\[
		\cond^{5-\d} \le \frac{1}{C_\d m(L)} \le \frac{103680}{C_\d}.
	\]
	Choosing $\d=\frac12$ and plugging in \eqref{eqn:C1/2bound} thus yields $\cond^{\frac{9}{2}}\leq 1595077$, which implies that $\cond\leq 23$.\\
	(5) Suppose that $\ell=7$. In this case, \cite[Table 2]{KL1} and the table in \cite[Subsubsection 7.1.4]{KL1} yield that if $L$ has class number one, then 
	\[
		\frac{1}{m(L)}\leq 2903040\quad\text{ and }\quad d_L\leq 2^{18}3^{6}.
	\]
	Thus if $h_{L+\frac{\bm{u}}{\cond}}=1$, then \eqref{eqn:hlowereval} implies that
	\[
		\cond^{6-\d} \le \frac{1}{C_\d m(L)} \le \frac{2903040}{C_\d}.
	\]
	Choosing $\d=\frac12$ and plugging in \eqref{eqn:C1/2bound} thus yields $\cond^{\frac{11}{2}}\leq 44662154$, which implies that $\cond\leq  23$.\\
	(6) Suppose that $\ell=8$. In this case, \cite[Table 2]{KL1} and the table in \cite[Subsubsection 7.1.3]{KL1} yield that if $L$ has class number one, then 
	\[
		\frac{1}{m(L)}\leq 696729600\quad\text{ and }\quad d_L\leq 3^{14}.
	\]
	Thus if $h_{L+\frac{\bm{u}}{\cond}}=1$, then \eqref{eqn:hlowereval} implies that
	\[
		\cond^{7-\d} \le \frac{1}{C_\d m(L)} \le \frac{696729600}{C_\d}.
	\]
	Choosing $\d=\frac12$ and plugging in \eqref{eqn:C1/2bound} thus yields $\cond^{\frac{13}{2}}\leq 10718916924$ which implies that $\cond\leq 33$.\\
	(7) Suppose that $\ell=9$. In this case, \cite[Table 2]{KL1} and the table in \cite[Subsubsection 7.1.2]{KL1} yield that if $L$ has class number one, then 
	\[
		\frac{1}{m(L)}\leq 1393459200\quad\text{ and }\quad d_L\leq 2^{24}.
	\]
	Thus if $h_{L+\frac{\bm{u}}{\cond}}=1$, then \eqref{eqn:hlowereval} implies that
	\[
		\cond^{8-\d} \le \frac{1}{C_\d m(L)} \le \frac{1393459200}{C_\d}.
	\]
	Choosing $\d=\frac12$ and plugging in \eqref{eqn:C1/2bound} thus yields $\cond^{\frac{15}{2}}\leq 21437833847$, which gives that $\cond\leq 23$.\\
	(8) Finally suppose that $\ell=10$. In this case, from the table in \cite[Subsubsection 7.1.1]{KL1} there are two class number one lattices $L$ with determinant $d_L=3^9$ and 
	\[
		\frac{1}{m(L)}=8360755200.
	\]
	Thus if $h_{L+\frac{\bm{u}}{\cond}}=1$, then \eqref{eqn:hlowereval} implies that
	\[
		\cond^{9-\d} \le \frac{1}{C_\d m(L)} \le \frac{8360755200}{C_\d}.
	\]
	Choosing $\d=\frac12$ and plugging in \eqref{eqn:C1/2bound} yields $\cond^{\frac{17}{2}}\leq 128627003077$, which implies that $\cond\leq 19$.
\end{proof}


\end{document}